\documentclass[12pt,twoside]{article}

\usepackage{a4}

\usepackage{amssymb,amsmath,amsthm,latexsym}
\usepackage{amsfonts}
\usepackage{amsfonts}
\usepackage{graphicx}

\numberwithin{subcase}{case}

\usepackage{amsmath, amsfonts}
\usepackage{amssymb, graphicx}
\usepackage{amscd}
\usepackage{textcomp}
\usepackage{palatino}
\usepackage{xcolor}
\usepackage{array}
\usepackage[colorlinks=true,linkcolor=red,citecolor=red]{hyperref}
\allowdisplaybreaks

\newtheorem{theorem}{Theorem}[section]

\newtheorem{corollary}[theorem] {Corollary}

\newtheorem{example}[theorem]{Example}
\newtheorem{lemma} [theorem]{Lemma}

\newcommand{\Z}{{\mathbb Z}}

\newcommand{\Q}{{\mathbb Q}}

\newcommand{\N}{{\mathbb N}}

\newcommand{\F}{{\mathbb F}}

\newtheoremstyle{case}{}{}{}{}{}{:}{ }{}
\theoremstyle{case}

\usepackage{indentfirst}

\vspace{5cm}

\begin{document}
	\label{'ubf'}  
	\setcounter{page}{1} 
	\markboth
	{\hspace*{-9mm} \centerline{\footnotesize
			Characterization of Monogenic Number Fields}}
	{\centerline{\footnotesize 
			Characterization of Monogenic Number Fields
		} \hspace*{-9mm}}
	
	\vspace*{-2cm}
	\begin{center}
		
		{\textbf{On characterization of Monogenic number fields associated with certain
		quadrinomials and its applications}\\
			\vspace{.2cm}
			\medskip
			{\sc Tapas Chatterjee}\\
			{\footnotesize  Department of Mathematics,}\\
			{\footnotesize Indian Institute of Technology Ropar, Punjab, India.}\\
			{\footnotesize e-mail: {\it tapasc@iitrpr.ac.in}}
			
			\medskip
			{\sc Karishan Kumar}\\
			{\footnotesize Department of Mathematics, }\\
			{\footnotesize Indian Institute of Technology Ropar, Punjab, India.}\\
			{\footnotesize e-mail: {\it karishan.22maz0012@iitrpr.ac.in}}
			\medskip}
		
	\end{center}
	\thispagestyle{empty} 
	\vspace{-.4cm}
	
	\begin{abstract}  
		{\footnotesize }
		Let $f(x)=x^{n}+ax^{3}+bx+c$ be the minimal polynomial of an algebraic integer $\theta$ over the rationals with certain conditions on $a,~b,~c,$ and $n.$ Let $K=\Q(\theta)$ be a number field and $\mathcal{O}_{K}$ be the ring of integers of $K.$ In this article, we characterize all the prime divisors of the discriminant of $f(x)$ which do not divide the index of $\theta.$ 
		As an interesting corollary, we establish necessary and sufficient conditions for $\Z[\theta]$ to be integrally closed. 
		Finally, we investigate the types of solutions to certain differential equations associated with the polynomial $f(x).$
	\end{abstract}
		
		
		\vspace{-.37cm}
		
		\section{\bf Introduction}
		 Let $\theta$ be an algebraic integer and $f(x)$ be the minimal polynomial of $\theta$ over the field of rational numbers. Let $K=\mathbb{Q}(\theta)$ be an algebraic number field of degree $n.$ Then, the field $K$ is called monogenic if it has an integral basis of the form $\{1, \theta, \theta^{2}, \theta^{3}, \ldots, \theta^{n-1}\}.$ The basis $\{1, \theta, \theta^{2}, \theta^{3}, \ldots, \theta^{n-1}\}$ is known as power basis of $K.$ 
		  An important criterion established by Dedekind in 1878 provides necessary and sufficient conditions for a prime $p$ to divide the index $[\mathcal{O}_{K}:\mathbb{Z}[\theta]]$ (Theorem \ref {T4}), where $\mathcal{O}_{K}$ is the ring of algebraic integers of the number field $K.$ This criterion and the Dedekind's renowned formula $$D_{f}= [\mathcal{O}_{K}:\mathbb{Z}[\theta]]^{2}D_{K},$$ play pivotal role in assessing the monogenity of the number field $K,$ where $D_{f}$ stands for the discriminant of $f(x),$ $D_{K}$ denotes the discriminant of the field $K,$ and $[\mathcal{O}_{K}:\mathbb{Z}[\theta]]$ represents the index of $\mathbb{Z}[\theta]$ in $\mathcal{O}_{K}.$ \\
		\indent In recent years, numerous mathematicians \cite{CK1, le, jss, LJ} have obtained results concerning the monogenity of algebraic number fields associated with trinomials and a particular class of quadrinomials. Lately, the authors \cite{CK1, CK2} have investigated the monogenity of some number fields associated with certain type of quadrinomials of the form $$x^{n}+ax^{n-1}+bx+c\in \mathbb{Z}[x].$$
		
		In this article, we use the Dedekind criterion to characterize all the prime divisors of $[\mathcal{O}_{K}:\mathbb{Z}[\theta]],$ where $\theta$ is the root of the irreducible polynomial $$f(x)=x^{n}+ax^{3}+bx+c\in \mathbb{Z}[x]$$ with $\frac{a}{a-c}=k\in\N$ such that $n=3k>4,$ and $2ab=3ac-bc.$\\ 
		\indent As a consequence of the above criterion, we get necessary and sufficient conditions to verify the associated number field of the polynomial $f(x)$ is monogenic or not, which depends only on $a,~b,~c,$ and $n.$ Equivalently, we verify if the set$\{1, \theta, \theta^{2}, \theta^{3}, \ldots, \theta^{n-1}\}$ is an integral basis  of $K$ or not. Additionally, due to our main Theorem (\ref {th1}), we have an interesting result that describes the structure of the solutions to a class of differential equations that are related to the polynomials of the form $f(x).$ The final section of the article contains some examples which state the importance of the given theorems. In this article, $\bar \nu(x)$ denotes the operation reduction modulo $p$ for any polynomial $\nu(x)$ and rational prime $p.$
		In this regard, we have the following theorem.
		\begin{theorem}\label{th1}
		Let $K=\mathbb{Q}(\theta)$ be a number field and $\theta$ be an algebraic integer with the minimal polynomial $$f(x)=x^{n}+ax^{3}+bx+c$$ over the field $\mathbb{Q},$ where  $\frac{a}{a-c}=k\in\N$ such that $n=3k>4,$ and $2ab=3ac-bc.$ Let $\mathcal{O}_{K}$ be the ring of algebraic integers of $K.$ A prime factor $p$ of the discriminant $D_{f}$ of $f(x)$ does not divide $[\mathcal{O}_{K}:\mathbb{Z}[\theta]]$ if and only if $p$ satisfies one of the following conditions:  
			\begin{enumerate}
				\item When $p|a,$ $p|b$ and $p|c,$ then $p^{2}\nmid c.$
				\item When $p | a,$ $p | b$ and $p\nmid c$ with  $u_{0}=\frac{c+(-c)^{p^{r}}}{p},$ $u_{1}=\frac{b}{p},$ $u_{2}=\frac{a}{p},$ and $p^{r}||n,$ then one of the following conditions is satisfied:\\
				(i) Exactly two elements in the set $\{ u_{0},  u_{1},  u_{2}\}$ are divisible by $p.$\\
				(ii) If $ p\nmid u_{0} u_{1}$ and $p|u_{2},$ then $$(- u_{0})^{n}+(u_{1})^{n}{c}\not \equiv 0\pmod{p}.$$ 
				(iii) If $ p\nmid u_{0} u_{2}$ and $p|u_{1},$ then $${c}(  u_{2})^{\frac {n}{3}}+( -u_{0})^{\frac {n}{3}} \not \equiv 0\pmod{p}.$$
				(iv)  If $ p\nmid u_{1} u_{2}$ and $p|u_{0},$                                                                                                                                                                                                                                                                                                                                                                                                                                                                                                                                                                                                                                                                                                                                                                                                                                                                                                                                                                                                                                                                                                                                                                                                                                                                                                                                                                                                                                                                                                                                                                                                                                                                                                                                                                                                                                                                                                                                                                                                                                                                                                                                                                                                                                                                                                                                                                                                                                                                                                                                                                                                                                                                                                                                                                                                                                                                                                                                                                                                                                                                                                                                                                                                                                                                                                                                                                                                                                                                                                                                                                                                                                                                                                                                                                                                                                                                                                                                                                                                                                                                                                                                                                                                                                                                                                                                                                                                                                                                                                                                                                                                                                                                                                                                                                                                                                                                                                                                                               then $$\left\{\begin{array}{ll}
					{c}( u_{2})^{\frac{n}{2}}+(- u_{1})^{\frac{n}{2}}\not\equiv 0\pmod{p}, & {\rm if }\; 2|n \\
					{c}^{2}( u_{2})^{n}+( u_{1})^{n}\not\equiv 0\pmod{p}, & {\rm if}\;2\nmid n 
				\end{array}\right..$$
				(v) If $ p\nmid (u_{0} u_{1}u_{2}),$ then $$[(-\bar {c})^{\frac {3}{n}}\bar u_{2}+\bar u_{0}]^{3}+(\bar u_{1})^{3}(-\bar c)^{\frac {3}{n}}\neq\bar 0.$$
				
				\item	When $p|a,$ $p\nmid b$ and $p|c$ with $v_{0}=\frac{c}{p},$ $v_{1}=\frac{b+(-b)^{p^{s}}}{p},$ $v_{2}=\frac{a}{p},$ and $p^{s}||(n-1),$ then one of the following conditions is satisfied:
				
				(i) $ p|v_{2}$ and $v_{0}[(- v_{0})^{n-1}+(v_{1})^{n-1}{b}]\not \equiv 0\pmod{p}.$
				 
				(ii) If $ p\nmid v_{0} v_{2}$ and $p|v_{1},$ then $$
				( v_{2})^{n-1} b^{3}+(-v_{0})^{n-1}\not\equiv 0\pmod{p}.$$ 
				 
				(iii) If $ p\nmid (v_{0} v_{1}v_{2}),$ then $$[(-\bar {b})^{\frac {3}{n-1}}\bar v_{2}+\bar v_{0}]^{3}+(\bar v_{1})^{3}(-\bar b)^{\frac {3}{n-1}}\neq\bar 0.$$
				
				\item When $p\nmid a,$ $p|b$ and $p|c,$ with $w_{0}=\frac{c}{p},$ $w_{1}=\frac{b}{p},$ $w_{2}=\frac{a+(-a)^{p^{t}}}{p},$ and $p^{t}||(n-3),$ then one of the following conditions is satisfied:
				
				(i) $ p|w_{2}$ and $w_{0}[(- w_{0})^{n-3}+(w_{1})^{n-3}{a}]\not \equiv 0\pmod{p}.$
				
				(ii) If $ p\nmid w_{0} w_{2}$ and $p|w_{1},$ then $$
				(- w_{0})^{\frac {n-3}{3}}+ {a}(w_{2})^{\frac {n-3}{3}}\not\equiv 0\pmod{p}.$$
				(iii) If $ p\nmid (w_{0} w_{1}w_{2}),$ then $$[(-\bar {a})^{\frac {3}{n-3}}\bar w_{2}+\bar w_{0}]^{3}+(\bar w_{1})^{3}(-\bar a)^{\frac {3}{n-3}}\neq\bar 0.$$
				
				\item When $p|b$ and $p\nmid ac$ with $n=3k=3(3k_{1}+R),$ then one of the following conditions is satisfied:\\
				(i) If $R=1,$ then $9\nmid (a^{2}-1).$
				
				(ii) If $R=2,$ then $$\bigg[~~ \overline {2\bigg(\frac {a^{2}-1}{3}\bigg)}x^{3}+ x^{2}+\overline {(2+ac)}x+\overline {\bigg(\frac {c^{2}-1}{3}\bigg)}\bigg]$$ is co-prime to $\bar f(x).$
				
				\item When $p\nmid abc,$ then one of the following conditions is satisfied:\\
				(i) $(x^{3}+x+1)$ is co-prime to $\bar f(x).$
				
				(ii) If $(x^{3}+x+1)$ is not co-prime to $\bar f(x),$ then either $(x^{3}+x+1)$ is co-prime to $\bar M(x)$ or $\bar M(\zeta)\neq \bar 0,$ where $$M(x)=\frac{1}{p}\bigg(f(x)-\displaystyle\prod_{i=1}^{l_{5}}  q_{i}(x)^{e_{i}}\bigg)$$ given $\bar f(x)=\displaystyle\prod_{i=1}^{l_{5}} (\bar q_{i}(x))^{e_{i}}$ with $\bar q_{i}(x)$ are the distinct monic irreducible polynomial factors of $\bar f(x)$ over the field $\F_{p}$ and $q_{i}(x)$ are  respectively monic lifts, for all $i\in\{1,~2,~\ldots,l_{5}\},$ and $\zeta$ is a common zero of $(x^{3}+x+1)$ and $\bar f(x).$
				
			 \end{enumerate}
		\end{theorem}
	As a consequence of Theorem (\ref {th1}), we have the following important corollary. 
	\begin{corollary}\label{C1}
		Let $K=\mathbb{Q}(\theta)$ be the number field corresponding to the minimal polynomial $f(x)= x^{n}+ax^{3}+bx+c$ of the algebraic integer $\theta.$ Then, $\mathcal{O}_{K}=\Z[\theta]$ if and only if each prime $p$ dividing the discriminant of $f(x)$ satisfies any one of the conditions (1) to (6) of~ Theorem (\ref{th1}).
	\end{corollary}
	
		Further, from the above Theorem (\ref {th1}), we have the following important theorem.
		
		\begin{theorem}\label{th2}
			Let	\begin{align}\label{eq300}
				\frac{d^ny}{dx^n} + a \frac{d^3y}{dx^3}+b\frac{dy}{dx} + c y = 0
			\end{align}
			be a differential equation with integer coefficients, where $\frac{a}{a-c}=k\in\N$ such that $n=3k>4,$ and $2ab=3ac-bc.$ Let $\phi(z)=z^{n}+az^{3}+bz+c$ be the auxiliary polynomial of (\ref {eq300}) which is irreducible with a root $\theta.$ If for each prime $p$ dividing the discriminant $D_{\phi} $ of $\phi (z)$ satisfies any one of the conditions (1) to (6) of~ Theorem (\ref{th1}), then the general solution of the given differential equation (\ref {eq300}) is of the form
			\begin{align}
				\label{eq301}
				y(x)&=
				\sum_{i=1}^{n}c_{i}\displaystyle\prod_{j=1}^{n} e^{a_{j-1}^{(i)}\theta^{j-1}x},
			\end{align}
			where $a_{j-1}^{(i)}$ are integers and $c_{i}$ are arbitrary real constants, for all $i,~j\in \{1,~2,~\ldots,~n\}.$
		\end{theorem}

		\section{\bf Notations and Preliminaries}
		In this section, we define some basic notations and results.
		Let $p$ be a prime number and $n$ be any integer such that $p\nmid n,$ then $(\bar n)^{-1}$ denotes the inverse of $ n$ in the field $\F_{p}.$
		We need the following results that play important roles in the proof of Theorem (\ref {th1}). 
		
		\begin{lemma}\label{l1}
		If $n=\frac{3a}{a-c} \in\N $ and $2ab=3ac-bc,$ then $nc=(n-3)a=(n-1)b.$
		\end{lemma}
		\begin{proof}
			Let $n=\frac{3a}{a-c}.$ Then, we have 
			\begin{equation}\label{eq1}
			n(a-c)=3a~~ \text {if and only if} ~~	nc=(n-3)a.
			\end{equation}
		Substituting $c=\frac{(n-3)a}{n}$ in $2ab=3ac-bc,$ we get
			\begin{align}\label{eq2}
			(3a-b)\frac{(n-3)a}{n}=2ab&~~ \text {if and only if} ~~	(3a-b)\bigg(1-\frac{3}{n}\bigg)=2b\nonumber \\&~~ \text {if and only if} ~~a\bigg(1-\frac{3}{n}\bigg)=b-\frac{b}{n}\nonumber\\&~~ \text {if and only if} ~~a\bigg(\frac{n-3}{n}\bigg)=b\bigg(\frac{n-1}{n}\bigg) \nonumber
			\\&~~ \text {if and only if}~~(n-3)a=(n-1)b.
			\end{align}
		By using equations (\ref{eq1}) and (\ref{eq2}), we obtain $nc=(n-1)b=(n-3)a.$
		\end{proof}
		The following lemma outlines the existence of primes under some special conditions.

		\begin{lemma}\label{l2}
			Let $k=\frac{a}{a-c}\in\N$ such that $n=3k,$ and $2ab=3ac-bc.$ Then the following results hold:  
				\begin{enumerate}
			\item There does not exist any prime $p$ which satisfies at least one of the following conditions:
			  
			(i)$p|a$ and $p\nmid bc$
			 
			 (ii) $p|c$ and $p\nmid ab.$
			\item The only possible prime which satisfies the conditions $p|b$ and $p\nmid ac$ is $3.$ In addition, if $\frac {a}{a-c}\in\N,$ then $3||n.$
			
				\end{enumerate}
	\end{lemma}	
\begin{proof}
	\begin{enumerate}
	\item If $n=\frac{3a}{a-c} \in\N $ and $2ab=3ac-bc,$ then using Lemma (\ref {l1}), we have
		\begin{align}\label{eq3}
	 nc=(n-1)b=(n-3)a.
	\end{align}
(i) Let $p|a,$ then from equation (\ref {eq3}), we get $p|nc$ and $p|(n-1)b.$ Now, if $p\nmid bc,$ then $p|n$ and $p|(n-1)$ which is a contradiction to the fact that a prime number cannot divide two consecutive integers. Thus, there is no such prime exists. 

(ii) Now, if $p|c$ and $p\nmid ab,$ then from equation (\ref {eq3}), we get $p|(n-1)$ and $p|(n-3)$ which implies that $p|[(n-1)-(n-3)]$ i.e. $p|2.$ Thus, $p=2.$ As $p=2$ and $p\nmid ab,$ implies that $a$ and $b$ are odd numbers. This gives us $4|c(3a-b)$ i.e. $4|2ab$ (since $2ab =3ac-bc$) which is not possible because $2\nmid ab.$ 

 This completes the proof of (1).
\item
 If $p|b$ and $p\nmid ac,$ then by using equation (\ref {eq3}), we obtain $p|n$ and $p|(n-3)$ which implies that $p|[n-(n-3)]$ i.e. $p|3.$ Hence, $p=3.$ Again, from equation (\ref {eq3}), we have $nc=(n-3)a$ i.e. $a=\frac{n}{n-3}c.$ Let $n=3k,$ then $\frac{n}{n-3}=\frac{k}{k-1},$ where $k\in\N.$ This gives us that if $9|n,$ then $3|a$ which is a contradiction to the fact that $p\nmid ac.$ This completes the proof of the second part.
 
	\end{enumerate}
\end{proof}
The following lemma plays a crucial role in establishing our main theorem.

\begin{lemma}\label{l3}
	Let $A(x)=x^{q}+d\in\Z[x]$ be any polynomial of degree $q,$ where $q=p^{k}m,$ $p\nmid md$ for some prime $p,$ and $k,~m\in\N.$ Then, there exist two polynomial functions $B(x)$ and $C(x)$ such that $$A(x)= \bigg(\displaystyle\prod_{i=1}^{l} g_{i}(x)\bigg)^{p^{k}}+p\bigg(\displaystyle\prod_{i=1}^{l} g_{i}(x) \bigg)
	B(x)+p^{2}C(x)+(d+(-d)^{p^{k}}),$$ where polynomials $g_{i}(x)$ are monic lifts of the polynomials $\bar g_{i}(x),$ which are the distinct monic irreducible polynomial factors of $x^{m}+\bar {d}$ over the field $\F_{p},$ for all $i=1,~~2,\ldots,l.$
	\end{lemma}
\begin{proof}
	Let $q=p^{k}m$ and $p^{k}||q,$ $k\in\N.$ Consider, $A(x)=x^{q}+d=x^{p^{k}m}+d.$ 
Let $D(x)=x^{m}+d$ and applying reduction modulo $p$ operation on $D(x),$ we have $\bar D(x)=x^{m}+\bar d.$ On differentiating $\bar D(x),$ we have $\bar D'(x)=mx^{m-1}.$ As we know, a polynomial has repeated zero if and only if its derivative vanishes at that zero. Consequently, $\bar D(x)$ has no repeated zeros. Let $\displaystyle\prod_{i=1}^{l} \bar g_{i}(x)$ be the factorization of $x^{m}+\bar {d}$ over the field $\F_{p},$ where $\bar g_{i}(x)$ are distinct monic irreducible polynomials and $g_{i}(x)$ are respectively monic lifts. We can write 
\begin{align}
	\label{eq7}
	x^{m}+d= \displaystyle\prod_{i=1}^{l} g_{i}(x)+pU(x),
\end{align}
for some $U(x)\in \mathbb{Z}[x].$ 
  From equation (\ref {eq7}), substituting the value of $x^{m}$ in $A(x),$ we obtain
  \begin{align}
  	\label{eq8}
  	A(x)=x^{p^{k}m}+d= \bigg(\displaystyle\prod_{i=1}^{l} g_{i}(x)+pU(x)-d\bigg)^{p^{k}}+d.
  \end{align}
Using binomial theorem, we have from equation (\ref{eq8}),
\begin{align}
	\label{eq9}
	A(x)&= \bigg(\displaystyle\prod_{i=1}^{l} g_{i}(x)+pU(x)-d\bigg)^{p^{k}}+d\nonumber\\ &=
	\sum_{r=0}^{p^{k}}\binom {p^{k}}{r}\bigg(\displaystyle\prod_{i=1}^{l} g_{i}(x)-d\bigg)^{p^{k}-r}[pU(x)]^{r}+d\nonumber\\&=\bigg(\displaystyle\prod_{i=1}^{l} g_{i}(x)-d\bigg)^{p^{k}}+\sum_{r=1}^{p^{k}}\binom {p^{k}}{r}\bigg(\displaystyle\prod_{i=1}^{l} g_{i}(x)-d\bigg)^{p^{k}-r}[pU(x)]^{r}+d\nonumber\\&=\sum_{r=0}^{p^{k}}\binom {p^{k}}{r}\bigg(\displaystyle\prod_{i=1}^{l} g_{i}(x)\bigg)^{p^{k}-r}(-d)^{r}+p^{2}C(x)(\text {say})+d\nonumber\\&=\bigg(\displaystyle\prod_{i=1}^{l} g_{i}(x)\bigg)^{p^{k}}+\sum_{r=1}^{p^{k}-1}\binom {p^{k}}{r}\bigg(\displaystyle\prod_{i=1}^{l} g_{i}(x)\bigg)^{p^{k}-r}(-d)^{r}+(-d)^{p^{k}}+p^{2}C(x)+d\nonumber\\&=\bigg(\displaystyle\prod_{i=1}^{l} g_{i}(x)\bigg)^{p^{k}}+p\bigg(\displaystyle\prod_{i=1}^{l} g_{i}(x) \bigg)
	B(x)(\text {say})+p^{2}C(x)+(-d)^{p^{k}}+d,
\end{align}
where $B(x)$ and $C(x)$ contains other remaining terms. This completes the proof.
\end{proof}
In 1878, Dedekind introduced a notable criterion known as the Dedekind criterion (\cite{HC}, Theorem 6.1.4; \cite{RD}). This criterion provides the necessary and sufficient conditions that the polynomial $f(x)$ must satisfy to ensure that a prime number $p$ does not divide the index $[\mathcal{O}_{K}:\mathbb{Z}[\theta]].$
	
		\begin{theorem}\label{T4}
			(Dedekind Criterion) Let $\theta$ be an algebraic integer and $f(x)$ be the minimal polynomial of $\theta$ over $\mathbb{Q}.$ Let $K=\mathbb{Q}(\theta)$ be the corresponding number field. Let $p$ be a prime and $$\bar{f}(x)=\bar {f}_1(x)^{a_1}\bar {f}_2(x)^{a_2}\cdots\bar {f}_t(x)^{a_t}$$ be the factorization of $\bar{f}(x)$ as a product of powers of distinct monic irreducible polynomials over the field $\F_{p}.$ Let $M(x)$ be the polynomial defined as $$M(x) = \frac{1}{p}(f(x)-{f}_1(x)^{a_1} {f}_2(x)^{a_2} \cdots {f}_t(x)^{a_t})\in\mathbb{Z}[x],$$  where $f_i(x)\in \mathbb{Z}[x]$ are monic lifts of $\bar f_i(x),$ for all $i=1,~2,\cdots,t.$ Then, a prime $p \nmid [\mathcal{O}_{K}:\mathbb{Z}[\theta]]$ if and only if for each $i,$ we have either $a_i=1$ or $\bar{f_i}(x)$ does not divide $\bar{M}(x).$   
		\end{theorem}
	Now, we present a lemma that generalizes the first part of Theorem (\ref{th1}).
		
		\begin{lemma}\label{L5}
			Let $\theta$ be an algebraic integer, $n\geq 2$ be any integer, and $$f(x)= x^{n}+ c_{n-1}x^{n-1}+ c_{n-2}x^{n-2}+\cdots+c_{1}x+ c_{0},$$ be the minimal polynomial of $\theta$ over $\mathbb{Q}.$ Let $K=\mathbb{Q}(\theta)$ be the corresponding number field. Let $p$ be a prime number which divides $c_{i},$ for all $i= 0, 1, 2,\ldots,(n-1).$ Then, $p\nmid [\mathcal{O}_{K}:\mathbb{Z}[\theta]]$ if and only if $p^{2}\nmid  c_{0}.$ 
		\end{lemma}
		\begin{proof}
			Let $p|c_{i},$ for all $i= 0, 1,2, \ldots,(n-1),$ where $p$ be any prime number. Then, $$f(x)= x^{n}+ c_{n-1}x^{n-1}+ c_{n-2}x^{n-2}+\cdots+c_{1}x+ c_{0} \equiv x^{n} \pmod{p}$$ which implies that $$\bar {f}(x)=x^{n}\in \F_{p}[x].$$ Since $n\geq 2,$ by Dedekind criterion, $p\nmid [\mathcal{O}_{K}:\mathbb{Z}[\theta]]$ if and only if $x$ does not divide $\bar {M}(x),$ where $$M(x)=\frac{c_{n-1}x^{n-1}+ c_{n-2}x^{n-2}+\cdots+c_{1}x+ c_{0}}{p}.$$ Here, $x$ divides $\bar {M}(x)$ if and only if $p^{2}|c_{0}$ or we can say that $x$ does not divide $\bar {M}(x)$ if and only if $p^{2}\nmid c_{0}.$ Thus, $p\nmid [\mathcal{O}_{K}:\mathbb{Z}[\theta]]$ if and only if $p^{2}\nmid c_{0}.$ This completes the proof.   
		\end{proof}

		\section{\bf Proofs of the main theorem}
		
		\begin{proof}[\bf{Proof of Theorem \ref{th1}}]
			We prove each part of the theorem separately.
			Consider the $\textbf{first part},$ when $p|a,$ $p|b,$ and $p|c,$ where $p$ is any prime. By substituting $c_{0}=c,$ $c_{1}=b,$ $c_{3}=a,$ and $c_{i}=0,$ for all remaining $i=2,~4,~5,~\ldots,(n-1)$ in Lemma (\ref{L5}), we have complete proof of the first part of the theorem directly.
			
			\indent	Now, consider the $\textbf{second part}$ when $p|a,$ $p|b,$ and $p\nmid c.$ By using Lemma (\ref {l1}), we get $p|n$ that means there exist two positive integers $r$ and $m$ such that $n=p^{r}m$ and $p\nmid m.$ Further, we have $$f(x)= x^{n}+ax^{3}+bx+c\equiv x^{n}+c \pmod{p},$$ i.e. $$\bar {f}(x)=x^{n}+\bar {c}=x^{p^{r}m}+\bar {c}\in \F_{p}[x].$$ 
			Since $p\nmid c$ implies that $\gcd (p, c)=1,$ and using Fermat's little theorem, we obtain $$c^{p^{r}} \equiv c \pmod{p}.$$ Thus, we have $$ {f}(x) \equiv (x^{m}+c)^{p^{r}} \pmod{p},$$ thanks to the binomial theorem.

			Let $\displaystyle\prod_{i=1}^{l_{1}} \bar G_{i}(x)$ be the factorization of $x^{m}+\bar {c}$ over the field $\F_{p},$ where $\bar G_{i}(x)$ are distinct monic irreducible polynomials and $G_{i}(x)$ are monic lifts,  respectively.
			Now,
			\begin{align}\label{eq508}
				f(x)&=x^{n}+ax^{3}+bx+c\nonumber\\
				&= x^{p^{r}m}+{c}+ax^{3}+bx.
			\end{align} 
			By using Lemma (\ref {l3}), we obtain 
			\begin{align}\label{eq50}
		f(x)=\bigg(\displaystyle\prod_{i=1}^{l_{1}} G_{i}(x)\bigg)^{p^{r}}+p\bigg(\displaystyle\prod_{i=1}^{l_{1}} G_{i}(x) \bigg)
				V_{1}(x)+p^{2}V_{2}(x)+(c+(-c)^{p^{r}})+ax^{3}+bx,
			\end{align}
			where polynomials $V_{1}(x)$ and $V_{2}(x)$ include the remaining terms.
			Define $M(x)$ as $$M(x)=\frac{1}{p}\bigg(f(x)-\bigg(\displaystyle\prod_{i=1}^{l_{1}}  G_{i}(x)\bigg)^{p^{r}}\bigg).$$ Substituting the value of $f(x)$ from (\ref {eq50}) in $M(x),$ we get $$\bar M(x)= \bigg(\displaystyle\prod_{i=1}^{l_{1}} \bar G_{i}(x)\bigg)\bar {V_{1}}(x)+\bar u_{2}x^{3}+\bar {u_1}x+\bar {u_0},$$ where $u_{2}=\frac {a}{p},$ $u_{1}=\frac {b}{p},$ and $u_{0}=\frac{(c+(-c)^{p^{r}})}{p}.$
			Let $\eta$ be a common zero of $\bar f(x)$ and $\bar M_{1}(x)$ in the algebraic closure of the field $\F_{p},$ where 
			\begin{align}
				\label{eq276}
				M_{1}(x)=\bar u_{2}x^{3}+\bar {u_1}x+\bar {u_0}.
			\end{align}	
			This results in the following two equations
			\begin{align}
				\label{eq127}
				\bar f(\eta)=\eta^{n}+\bar {c}=\bar 0
			\end{align}	
			and
			\begin{align}
				\label{eq128}
				\bar M_{1}(\eta)=\bar u_{2}\eta^{3}+\bar u_{1}\eta+\bar u_{0}=\bar 0.
			\end{align}		
			From equation (\ref {eq276}), we have the following cases:
			
			\textbf{Case 2.1:}	
			If $\bar u_{2}=\bar 0,$ $\bar u_{1}=\bar 0$ and $\bar u_{0}=\bar 0,$ then $\bar M_{1}(x)=\bar 0$ which implies that $\bar G_{i}(x)| \bar M_{1}(x),$ for all $i=1, 2, \ldots, l_{1}.$
			
			\textbf{Case 2.2:}
			If $\bar u_{2}\neq \bar 0,$ $\bar u_{1}= \bar 0,$ $\bar u_{0}=\bar 0$ or $\bar u_{2}= \bar 0,$ $\bar u_{1}\neq \bar 0,$ $\bar u_{0}=\bar 0,$ then from equation (\ref {eq128}), we have $\eta=\bar 0$ but $\bar f(\bar 0)\neq \bar 0$ (since $p\nmid c$).	
			If $\bar u_{2}=\bar 0,$ $\bar u_{1}= \bar 0,$ $\bar u_{0}\neq \bar 0,$ then $\bar M_{1}(x)=\bar u_{0}\neq \bar 0.$
			Consequently, $\bar f(x)$ and $\bar M_{1}(x)$ have no common zeros implying that $\bar G_{i}(x)\nmid \bar M_{1}(x),$ for all $i=1, 2, \ldots, l_{1}.$
			
			\textbf{Case 2.3:}	
			If $\bar u_{2}=\bar 0,$ $\bar u_{1}\neq \bar 0,$ $\bar u_{0}\neq \bar 0,$ then from equation (\ref {eq128}), we have $\bar u_{1}\eta+\bar u_{0}=\bar 0$ or $\eta=-(\bar u_{1})^{-1}\bar u_{0}.$ Now, using the value of $\eta$ in the equation (\ref {eq127}), we get $$(-(\bar u_{1})^{-1}\bar u_{0})^{n}+\bar {c}=\bar 0~~\text{or}~~(-\bar u_{0})^{n}+(\bar u_{1})^{n}\bar {c}=\bar 0.$$ Thus, $\bar f(x)$ and $\bar M_{1}(x)$ have no common zeros if and only if  $$(- u_{0})^{n}+(u_{1})^{n}{c}\not \equiv 0\pmod{p}$$ which is further equivalent to $\bar G_{i}(x)\nmid \bar M_{1}(x),$ for all $i=1, 2, \ldots, l_{1}.$
			
			\textbf{Case 2.4:}
			If $\bar u_{2}\neq \bar 0,$ $\bar u_{1}= \bar 0,$ $\bar u_{0}\neq \bar 0,$ then from equation (\ref {eq128}), we have 
			\begin{align}
				\label {eq129}
				\eta^{3}=-(\bar u_{2})^{-1}\bar u_{0}.
			\end{align}
		Since $n=3k,$ therefore by 
		substituting the value of $\eta^{3}$ in the equation (\ref {eq127}), we get $$(-(\bar u_{2})^{-1}\bar u_{0})^{k}+\bar {c}=\bar 0~~\text{ or}~~ \bar {c}(\bar u_{2})^{k}+(-\bar u_{0})^{k}=\bar 0.$$ Thus, $\bar f(x)$ and $\bar M_{1}(x)$ have no common zeros if and only if 
				$$
				 {c}( u_{2})^{\frac{n}{3}}+(- u_{0})^{\frac{n}{3}}\not\equiv 0\pmod{p}$$ which is further equivalent to $\bar G_{i}(x)\nmid \bar M_{1}(x),$ for all $i=1, 2, \ldots, l_{1}.$ 
			
			\textbf{Case 2.5:}
			If $\bar u_{2}\neq \bar 0,$ $\bar u_{1}\neq \bar 0,$ $\bar u_{0}=\bar 0,$ then from equation (\ref {eq128}), we have
			\begin{align}
				\label{eq131}
				\eta^{2}=-(\bar u_{2})^{-1}\bar u_{1}.
			\end{align}
			Let $n=2T_{1}+R_{1},$ where $T_{1}\in\N$ and $R_{1}\in \{0,~1\}.$ 
		Substituting the value of $\eta^{2}$ in the equation (\ref {eq127}), we get $$\eta^{R_{1}}(-(\bar u_{2})^{-1}\bar u_{1})^{T_{1}}+\bar {c}=\bar 0~~\text{ or}~~ \eta^{R_{1}}=-\bar {c}(\bar u_{2}(-\bar u_{1})^{-1})^{T_{1}}.$$ If $R_{1}=0,$ then $\bar {c}(\bar u_{2})^{T_{1}}+(-\bar u_{1})^{T_{1}}=\bar 0.$ If $R_{1}=1,$ then putting the value of $\eta$ in (\ref {eq131}), we have $$(\bar {c})^{2}(\bar u_{2})^{n}= (- \bar u_{1})^{n}.$$ Thus, $\bar f(x)$ and $\bar M_{1}(x)$ have no common zeros if and only if 
		$$\left\{\begin{array}{ll}
			{c}( u_{2})^{\frac{n}{2}}+(- u_{1})^{\frac{n}{2}}\not\equiv 0\pmod{p}, & {\rm if }\; 2|n \\
			{c}^{2}( u_{2})^{n}+( u_{1})^{n}\not\equiv 0\pmod{p}, & {\rm if}\;2\nmid n 
		\end{array}\right.$$ which is further equivalent to $\bar G_{i}(x)\nmid \bar M_{1}(x),$ for all $i=1, 2, \ldots, l_{1}.$ 
		
			\textbf{Case 2.6:}
			If $\bar u_{2}\neq \bar 0,$ $\bar u_{1}\neq \bar 0,$ $\bar u_{0}\neq \bar 0,$ then from equation (\ref {eq128}), we have 
			\begin{align}
				\label{eq132}
				\eta^{3}=-(\bar u_{2})^{-1}(\bar u_{1}\eta+\bar u_{0}).
			\end{align}
			By putting the value of $\eta^{3}$ in the equation (\ref {eq127}), we have $$[-(\bar u_{2})^{-1}(\bar u_{1}\eta+\bar u_{0})]^{\frac {n}{3}}+\bar {c}=\bar 0.$$ On solving the above equation, we have $$\eta=-(\bar u_{1})^{-1}[(-\bar {c})^{\frac {3}{n}}\bar u_{2}+\bar u_{0}].$$ Now, using the value of $\eta$ in the equation (\ref {eq132}), we obtain $$[(-\bar {c})^{\frac {3}{n}}\bar u_{2}+\bar u_{0}]^{3}+(\bar u_{1})^{3}(-\bar c)^{\frac {3}{n}}=\bar 0.$$ Thus, $\bar f(x)$ and $\bar M_{1}(x)$ have no common zeros if and only if $$[(-\bar {c})^{\frac {3}{n}}\bar u_{2}+\bar u_{0}]^{3}+(\bar u_{1})^{3}(-\bar c)^{\frac {3}{n}}\neq\bar 0$$ which is further equivalent to $\bar G_{i}(x)\nmid \bar M_{1}(x),$ for all $i=1, 2, \ldots, l_{1}.$ 
			
			It is easy to see that $\bar G_{i}(x)\nmid \bar M(x)$ if and only if $\bar G_{i}(x)\nmid \bar M_{1}(x),$ for all $i=1, 2, \ldots, l_{1}.$ Thus, by considering all the above cases collectively and using the Dedekind  criterion (\ref {T4}), we complete the proof of the second part.

			\indent	Now, we deal with the $\textbf{third part}$ when $p|a,$ $p|c,$ and $p\nmid b.$ Using Lemma (\ref {l1}), we have $p|(n-1)$ that means there exist two positive integers $s$ and $m_{1}$ such that $(n-1)=p^{s}m_{1}$ and $p^{s}||(n-1) .$ Further, we observe that $$f(x)= x^{n}+ax^{3}+bx+c\equiv x^{n}+bx \pmod{p},$$ i.e. $$\bar {f}(x)=x^{n}+\bar {b}x=x(x^{p^{s}m_{1}}+\bar {b})\in \F_{p}[x].$$ 
			In similar to the previous part, using binomial theorem along with Fermat's little theorem, we obtain $$ {f}(x) \equiv x(x^{m_{1}}+b)^{p^{s}} \pmod{p}.$$
			
			Let $\displaystyle\prod_{i=1}^{l_{2}} \bar H_{i}(x)$ be the factorization of $x^{m_{1}}+\bar {b}$ over the field $\F_{p},$ where $\bar H_{i}(x)$ are distinct monic irreducible polynomials and $H_{i}(x)$ are  respectively monic lifts.
			We write
			\begin{align}\label{eq10}
				f(x)&=x^{n}+ax^{3}+bx+c\nonumber\\
				&= x(x^{p^{s}m_{1}}+b)+ax^{3}+c.
			\end{align} 
			By using Lemma (\ref {l3}), we get 
			\begin{align}\label{eq11}
				f(x)=x\bigg(\displaystyle\prod_{i=1}^{l_{2}} H_{i}(x)\bigg)^{p^{s}}+px\bigg(\displaystyle\prod_{i=1}^{l_{s}} H_{i}(x) \bigg)
				W_{1}(x)+p^{2}xW_{2}(x)+(b+(-b)^{p^{s}})x+ax^{3}+c,
			\end{align}
			where polynomials $W_{1}(x)$ and $W_{2}(x)$ include the remaining terms.
			Define $M(x)$ as $$M(x)=\frac{1}{p}\bigg(f(x)-x\bigg(\displaystyle\prod_{i=1}^{l_{2}}  H_{i}(x)\bigg)^{p^{s}}\bigg).$$ Putting the value of $f(x)$ from (\ref {eq11}) in $M(x),$ we have $$\bar M(x)= x\bigg(\displaystyle\prod_{i=1}^{l_{2}} \bar H_{i}(x)\bigg)\bar {W_{1}}(x)+\bar v_{2}x^{3}+\bar {v_1}x+\bar {v_0},$$ where $v_{2}=\frac {a}{p},$ $v_{1}=\frac {b+(-b)^{p^{s}}}{p},$ and $v_{0}=\frac{c}{p}.$ Here it is clear that, if $p^{2}|c,$ then $x$ divides both $\bar f(x)$ and $\bar M(x).$ Therefore, for the upcoming cases, we take $p^{2}\nmid c$ i.e. $\bar {v_0}\neq \bar 0.$ 
			Let $\xi\neq \bar 0$ be a common zero of $\bar f(x)$ and $\bar M_{2}(x)$ in the algebraic closure of the field $\F_{p},$ where 
			\begin{align}
				\label{eq12}
				M_{2}(x)=\bar v_{2}x^{3}+\bar {v_1}x+\bar {v_0}.
			\end{align}	
			From this, we have two following equations
			\begin{align}
				\label{eq13}
				\bar f(\xi)=\xi(\xi^{n-1}+\bar {b})=\bar 0
				~~\text {or}~~ \xi^{n-1}+\bar {b}=\bar 0
				\end{align}	
			and
			\begin{align}
				\label{eq14}
				\bar M_{2}(\xi)=\bar v_{2}\xi^{3}+\bar v_{1}\xi+\bar v_{0}=\bar 0.
			\end{align}		
			From equation (\ref {eq12}), we have the following cases:

			\textbf{Case 3.1:}
				If $\bar v_{2}=\bar 0,$ $\bar v_{1}= \bar 0,$ $\bar v_{0}\neq \bar 0,$ then $\bar M_{2}(x)=\bar v_{0}\neq \bar 0.$
			Thus, $\bar f(x)$ and $\bar M_{2}(x)$ have no common zeros implying that $\bar H_{i}(x)\nmid \bar M_{2}(x),$ for all $i=1, 2, \ldots, l_{2}.$
			
			\textbf{Case 3.2:}	
			If $\bar v_{2}=\bar 0,$ $\bar v_{1}\neq \bar 0,$ $\bar v_{0}\neq \bar 0,$ then from equation (\ref {eq14}), we have $\bar v_{1}\xi+\bar v_{0}=\bar 0$ or $\xi=-(\bar v_{1})^{-1}\bar v_{0}.$ Now, substituting the value of $\xi$ in the equation (\ref {eq13}), we get $$(-(\bar v_{1})^{-1}\bar v_{0})^{n-1}+\bar {b}=\bar 0~~\text{or}~~(-\bar v_{0})^{n-1}+(\bar v_{1})^{n-1}\bar {b}=\bar 0.$$ Thus, $\bar f(x)$ and $\bar M_{2}(x)$ have no common zeros if and only if  $$(- v_{0})^{n-1}+(v_{1})^{n-1}{b}\not \equiv 0\pmod{p}$$ which is further equivalent to $\bar H_{i}(x)\nmid \bar M_{2}(x),$ for all $i=1, 2, \ldots, l_{2}.$
			
			\textbf{Case 3.3:}
			If $\bar v_{2}\neq \bar 0,$ $\bar v_{1}= \bar 0,$ $\bar v_{0}\neq \bar 0,$ then equation (\ref {eq14}) gives us 
			\begin{align}
				\label {eq15}
				\xi^{3}=-(\bar v_{2})^{-1}\bar v_{0}.
			\end{align}
			Since $n=3k,$ therefore by 
			substituting the value of $\xi^{3}$ in the equation (\ref {eq13}), we get $$(-(\bar v_{2})^{-1}\bar v_{0})^{k}+\bar {b}\xi=\bar 0~~\text{ or}~~\xi=-(\bar {b})^{-1} (-(\bar v_{2})^{-1}\bar v_{0})^{k} .$$ Again putting the value of $\xi$ in (\ref {eq15}), we obtain $(\bar v_{2})^{n-1}(\bar b)^{3}\bar v_{0}=(-\bar v_{0})^{n}.$ Thus, $\bar f(x)$ and $\bar M_{2}(x)$ have no common zeros if and only if 
			$$
			( v_{2})^{n-1} b^{3}+(-v_{0})^{n-1}\not\equiv 0\pmod{p}$$ which is further equivalent to $\bar H_{i}(x)\nmid \bar M_{2}(x),$ for all $i=1, 2, \ldots, l_{2}.$ 
			
			\textbf{Case 3.4:}
			If $\bar v_{2}\neq \bar 0,$ $\bar v_{1}\neq \bar 0,$ $\bar v_{0}\neq \bar 0,$ then from equation (\ref {eq14}), we have 
			\begin{align}
				\label{eq17}
				\xi^{3}=-(\bar v_{2})^{-1}(\bar v_{1}\xi+\bar v_{0}).
			\end{align}
			By putting the value of $\xi^{3}$ in the equation (\ref {eq13}), we have $$[-(\bar v_{2})^{-1}(\bar v_{1}\xi+\bar v_{0})]^{\frac {n-1}{3}}+\bar {b}=\bar 0.$$ On solving the above equation, we have $$\xi=-(\bar v_{1})^{-1}[(-\bar {b})^{\frac {3}{n-1}}\bar v_{2}+\bar v_{0}].$$ Now, using the value of $\xi$ in the equation (\ref {eq17}), we obtain $$[(-\bar {b})^{\frac {3}{n-1}}\bar v_{2}+\bar v_{0}]^{3}+(\bar v_{1})^{3}(-\bar b)^{\frac {3}{n-1}}=\bar 0.$$ Thus, $\bar f(x)$ and $\bar M_{2}(x)$ have no common zeros if and only if $$[(-\bar {b})^{\frac {3}{n-1}}\bar v_{2}+\bar v_{0}]^{3}+(\bar v_{1})^{3}(-\bar b)^{\frac {3}{n-1}}\neq\bar 0$$ which is further equivalent to $\bar H_{i}(x)\nmid \bar M_{2}(x),$ for all $i=1, 2, \ldots, l_{2}.$ 
			
			It is very simple to check that $\bar H_{i}(x)\nmid \bar M(x)$ if and only if $\bar H_{i}(x)\nmid \bar M_{2}(x),$ for all $i=1, 2, \ldots, l_{2}.$ Thus, by considering all the above cases together and using the Dedekind  criterion (\ref {T4}), we complete the proof of the third part.

			\indent	Consider the $\textbf{fourth part}$ when $p\nmid a,$ $p| b,$ $p| c.$ 
			From Lemma (\ref {l1}), we have $p|(n-3)$ (since $p| b$ and $p\nmid a$) which implies that there exist two positive integers $t$ and $m_{2}$ such that $(n-3)=p^{t}m_{2}$ and $p^{t}||(n-3).$ Now, we have $$f(x)= x^{n}+ax^{3}+bx+c\equiv x^{n}+ax^{3} \pmod{p},$$ i.e. $$\bar {f}(x)=x^{n}+\bar ax^{3}=x^{3}(x^{p^{t}m_{2}}+\bar {a})\in \F_{p}[x].$$ 
			Using Fermat's little theorem with the binomial theorem, we get $$ {f}(x) \equiv x^{3}(x^{m_{2}}+ {a})^{p^{t}} \pmod{p}.$$
			
			Let $\displaystyle\prod_{i=1}^{l_{3}} \bar h_{i}(x)$ be the factorization of $x^{m_{2}}+\bar {a}$ over the field $\F_{p},$ where $\bar h_{i}(x)$ are distinct monic irreducible polynomials and $h_{i}(x)$ are  respectively monic lifts.
			Also,
			\begin{align}\label{eq18}
				f(x)&=x^{n}+ax^{3}+bx+c\nonumber\\
				&= x^{3}(x^{p^{t}m_{2}}+a)+bx+c.
			\end{align} 
			By applying Lemma (\ref {l3}), we obtain 
			\begin{align}\label{eq19}
				f(x)=x^{3}\bigg(\displaystyle\prod_{i=1}^{l_{3}} h_{i}(x)\bigg)^{p^{t}}+px^{3}\bigg(\displaystyle\prod_{i=1}^{l_{3}} h_{i}(x) \bigg)
				A_{1}(x)+p^{2}x^{3}A_{2}(x)+(a+(-a)^{p^{t}})x^{3}+bx+c,
			\end{align}
			where the polynomials $A_{1}(x)$ and $A_{2}(x)$ contain the remaining terms.
			Define $M(x)$ as $$M(x)=\frac{1}{p}\bigg(f(x)-x^{3}\bigg(\displaystyle\prod_{i=1}^{l_{3}}  h_{i}(x)\bigg)^{p^{t}}\bigg).$$ On substituting the value of $f(x)$ from equation (\ref {eq19}) in $M(x),$ we get $$\bar M(x)= x^{3}\bigg(\displaystyle\prod_{i=1}^{l_{3}} \bar h_{i}(x)\bigg)\bar {A_{1}}(x)+\bar w_{2}x^{3}+\bar {w_1}x+\bar {w_0},$$ where $w_{2}=\frac {a+(-a)^{p^{t}}}{p},$ $w_{1}=\frac {b}{p},$ and $w_{0}=\frac{c}{p}.$ It is easy to verify that if $p^{2}|c,$ then $x$ divides both $\bar f(x)$ and $\bar M(x).$ Therefore, we take $p^{2}\nmid c$ for the coming cases i.e. $\bar {w_0}\neq \bar 0.$ 
			Let $\alpha\neq \bar 0$ be a common zero of $\bar f(x)$ and $\bar M_{3}(x)$ in the algebraic closure of the field $\F_{p},$ where 
			\begin{align}
				\label{eq20}
				M_{3}(x)=\bar w_{2}x^{3}+\bar {w_1}x+\bar {w_0}.
			\end{align}	
			From this, we have following equations
			\begin{align}
				\label{eq21}
				\bar f(\alpha)=\alpha^{3}(\alpha^{n-3}+\bar {a})=\bar 0
				~~\text {or}~~ \alpha^{n-3}+\bar {a}=\bar 0
			\end{align}	
			and
			\begin{align}
				\label{eq22}
				\bar M_{3}(\alpha)=\bar w_{2}\alpha^{3}+\bar w_{1}\alpha+\bar w_{0}=\bar 0.
			\end{align}		
			From equation (\ref {eq20}), we have the following cases:

			\textbf{Case 4.1:}
			If $\bar w_{2}=\bar 0,$ $\bar w_{1}= \bar 0,$ $\bar w_{0}\neq \bar 0,$ then $\bar M_{3}(x)=\bar w_{0}\neq \bar 0.$
			Thus, $\bar f(x)$ and $\bar M_{3}(x)$ have no common zeros which implies that $\bar h_{i}(x)\nmid \bar M_{3}(x),$ for all $i=1, 2, \ldots, l_{3}.$
			
			\textbf{Case 4.2:}	
			If $\bar w_{2}=\bar 0,$ $\bar w_{1}\neq \bar 0,$ $\bar w_{0}\neq \bar 0,$ then from equation (\ref {eq22}), we have $\bar w_{1}\alpha+\bar w_{0}=\bar 0$ or $\alpha=-(\bar w_{1})^{-1}\bar w_{0}.$ Now, substituting the value of $\alpha$ in the equation (\ref {eq21}), we get $$(-(\bar w_{1})^{-1}\bar w_{0})^{n-3}+\bar {a}=\bar 0~~\text{or}~~(-\bar w_{0})^{n-3}+(\bar w_{1})^{n-3}\bar {a}=\bar 0.$$ Thus, $\bar f(x)$ and $\bar M_{3}(x)$ have no common zeros if and only if  $$(- w_{0})^{n-3}+(w_{1})^{n-3}{a}\not \equiv 0\pmod{p}$$ which is further equivalent to $\bar h_{i}(x)\nmid \bar M_{3}(x),$ for all $i=1, 2, \ldots, l_{3}.$
			
			\textbf{Case 4.3:}
			If $\bar w_{2}\neq \bar 0,$ $\bar w_{1}= \bar 0,$ $\bar w_{0}\neq \bar 0,$ then from equation (\ref {eq22}), we have 
			\begin{align}
				\label {eq23}
				\alpha^{3}=-(\bar w_{2})^{-1}\bar w_{0}.
			\end{align}
			Since $n=3k,$ therefore by 
			substituting the value of $\alpha^{3}$ in the equation (\ref {eq21}), we get $$(-(\bar w_{2})^{-1}\bar w_{0})^{k-1}+\bar {a}=\bar 0 $$ or $$(-\bar w_{0})^{k-1}+\bar {a}(\bar w_{2})^{k-1}=\bar 0 .$$ Thus, $\bar f(x)$ and $\bar M_{3}(x)$ have no common zeros if and only if 
			$$
			(- w_{0})^{\frac {n-3}{3}}+ {a}(w_{2})^{\frac {n-3}{3}}\not\equiv 0\pmod{p}$$ which is further equivalent to $\bar h_{i}(x)\nmid \bar M_{3}(x),$ for all $i=1, 2, \ldots, l_{3}.$ 
			
			\textbf{Case 4.4:}
			If $\bar w_{2}\neq \bar 0,$ $\bar w_{1}\neq \bar 0,$ $\bar w_{0}\neq \bar 0,$ then from equation (\ref {eq22}), we have 
			\begin{align}
				\label{eq24}
				\alpha^{3}=-(\bar w_{2})^{-1}(\bar w_{1}\alpha+\bar w_{0}).
			\end{align}
			By putting the value of $\alpha^{3}$ in the equation (\ref {eq21}), we have $$[-(\bar w_{2})^{-1}(\bar w_{1}\alpha+\bar w_{0})]^{\frac {n-3}{3}}+\bar {a}=\bar 0.$$ On solving the above equation, we get $$\alpha=-(\bar w_{1})^{-1}[(-\bar {a})^{\frac {3}{n-3}}\bar w_{2}+\bar w_{0}].$$ Now, using the value of $\alpha$ in the equation (\ref {eq24}), we obtain $$[(-\bar {a})^{\frac {3}{n-3}}\bar w_{2}+\bar w_{0}]^{3}+(\bar w_{1})^{3}(-\bar a)^{\frac {3}{n-3}}=\bar 0.$$ Thus, $\bar f(x)$ and $\bar M_{3}(x)$ have no common zeros if and only if $$[(-\bar {a})^{\frac {3}{n-3}}\bar w_{2}+\bar w_{0}]^{3}+(\bar w_{1})^{3}(-\bar a)^{\frac {3}{n-3}}\neq\bar 0$$ which is further equivalent to $\bar h_{i}(x)\nmid \bar M_{3}(x),$ for all $i=1, 2, \ldots, l_{3}.$ 
			
			It is easy to verify that $\bar h_{i}(x)\nmid \bar M(x)$ if and only if $\bar h_{i}(x)\nmid \bar M_{3}(x),$ for all $i=1, 2, \ldots, l_{3}.$ Thus, by considering all the above cases together and using the Dedekind  criterion (\ref {T4}), we complete the proof of the fourth part.
			
			\indent	Now consider the $\textbf{fifth part}$ when $p| b$ and $p\nmid ac.$ 
			From the second part of Lemma (\ref {l2}), we get $p=3.$ Now, $$f(x)= x^{n}+ax^{3}+bx+c\equiv x^{n}+ax^{3}+c \pmod{3},$$ i.e.
			
			\begin{equation}\label{eq27}
			 \bar {f}(x)=x^{n}+\bar ax^{3}+\bar c\in \F_{3} [x].
			\end{equation}
			Also according to the hypothesis $3|n$ and let $n=3k.$ Now, using Fermat's little theorem with the binomial theorem, we get $$ {f}(x) \equiv(x^{k}+ {a}x+c)^{3} \pmod{3}.$$
			
			Let $\displaystyle\prod_{i=1}^{l_{4}} \bar F_{i}(x)$ be the factorization of $x^{k}+ \bar{a}x+\bar c$ over the field $\F_{3},$ where $\bar F_{i}(x)$ are monic irreducible polynomials and $F_{i}(x)$ are  respectively monic lifts.
			Now, following the same steps of Lemma (\ref {l3}), we obtain 
			
\begin{equation}\label{eq25}
			\begin{split}
				f(x)&=\bigg(\displaystyle\prod_{i=1}^{l_{4}} F_{i}(x)\bigg)^{3}+3\bigg(\displaystyle\prod_{i=1}^{l_{4}} F_{i}(x) \bigg)
				B_{1}(x)+3^{2}C_{1}(x)+(a-a^{3})x^{3}\\&~~~~~~~~~~~~~~~~~~~~~~~~~~~~~~~~~~~~~~~~~~~~~~~~~~~~~~~~~-3a^{2}cx^{2}+(b-3ac^{2})x+(c-c^{3}) 
				\end{split}
			\end{equation}
			  	and
			  	\begin{equation}\label{eq28}
			  		\bar {f}(x)=\bigg(\displaystyle\prod_{i=1}^{l_{4}} \bar F_{i}(x)\bigg)^{3}\in \F_{3}[x].
			  	\end{equation}
			Define $M(x)$ as $$M(x)=\frac{1}{3}\bigg(f(x)-\bigg(\displaystyle\prod_{i=1}^{l_{4}}  F_{i}(x)\bigg)^{3}\bigg).$$ As $n=3k,$ therefore using the second part of Lemma (\ref {l2}), we have $k=3k_{1}+R$ and $R\in \{1,~2\}.$ On substituting $a=\frac {n}{n-3}c,$ $b=\frac {n}{n-1}c$ (\ref {l1}), and the value of $f(x)$ from equation (\ref {eq25}) in $M(x),$ we have
			\begin{align}\label{eq26}
				\begin{split}
			\bar M(x)&= \bigg(\displaystyle\prod_{i=1}^{l_{4}} \bar F_{i}(x)\bigg)\bar {B_{1}}(x)+\overline {\bigg(\frac {c}{(3k-1)(k-1)}\bigg)}\bigg[\overline {k(3k-1)} \overline {\bigg(\frac {1-a^{2}}{3}\bigg)}x^{3}\\&-\overline {a^{2}(3k-1)(k-1)}x^{2}+\overline {(k-1)(k-ac(n-1))}x+\overline {(3k-1)(k-1)}\overline {\bigg(\frac {c^{2}-1}{3}\bigg)}~\bigg]\\&= \bigg(\displaystyle\prod_{i=1}^{l_{4}} \bar F_{i}(x)\bigg)\bar {B_{1}}(x)+\overline {\bigg(\frac {c}{(3k-1)(k-1)}\bigg)}\bigg[~~ \overline {R\bigg(\frac {a^{2}-1}{3}\bigg)}x^{3}+ {(R-1)}x^{2}\\&~~~~~~~~~~~~~~~~~~~~~~~~~~~~~~~~~~~~~~~~~~~~~~~~~~~~~~~~~~~+(R-1)\overline {(R+ac)}x+(R-1)\overline {\bigg(\frac {c^{2}-1}{3}\bigg)}~~\bigg],
				\end{split}
		\end{align}
	(since Fermat's theorem gives us $a^{2}\equiv 1\pmod{3}$).
	Let $\beta$ be the common zero of $\bar f(x)$ and $\bar M(x)$ in the algebraic closure of the field $\F_{3}.$
		Now, we have two different cases according to the values of $R.$\\
		
		\textbf{Case 5.1:} Let $R=1.$ As $\beta$ is a zero of $\bar f(x),$ it satisfies some of the $\bar F_{i}(x)$ over the field $\F_{3}.$ Due to this, from the equations (\ref {eq27}) and (\ref {eq26}), we get 
			\begin{equation}\label{eq29}
				\bar {f}(\beta)=\beta^{n}+\bar a\beta^{3}+\bar c=\bar 0
			\end{equation}
			and
			\begin{equation}\label{eq30}
					\bar M(\beta)=\overline {\bigg(\frac {c}{(3k-1)(k-1)}\bigg)}\bigg[~~ \overline {\bigg(\frac {a^{2}-1}{3}\bigg)}\beta^{3}\bigg]=\bar 0~~\text {or}~~\overline {\bigg(\frac {a^{2}-1}{3}\bigg)}\beta^{3}=\bar 0.
				\end{equation}
		Now, we have two subcases due to equation (\ref {eq30}), which are as follows:
		
		\textbf{Subcase 5.1.1:} Let $9|(a^{2}-1).$ Then $$\bar M(x)= \bigg(\displaystyle\prod_{i=1}^{l_{4}} \bar F_{i}(x)\bigg)\bar {B_{1}}(x).$$ Thus, each $\bar F_{i}(x)|\bar M(x).$ Hence, by using the Dedekind  criterion (\ref {T4}), we have $3|[\mathcal{O}_{K}:\mathbb{Z}[\theta]]$ and conversely, if $3|[\mathcal{O}_{K}:\mathbb{Z}[\theta]],$ then $\bar F_{i}(x)|\bar M(x),$ for some $i.$

				\textbf{Subcase 5.1.2:} Let $9\nmid (a^{2}-1).$ Then, from equation (\ref {eq30}), we get $\beta=\bar 0$ but it is not possible because $\bar f(\bar 0)=\bar c\neq \bar 0.$ Thus, $\bar f(x)$ and $\bar M(x)$ have no common zeros
			i.e. $\bar F_{i}\nmid \bar M(x),$ for all $i=1, 2, \ldots, l_{4}.$ Therefore, by applying Dedekind  criterion (\ref {T4}), we have $3\nmid[\mathcal{O}_{K}:\mathbb{Z}[\theta]]$ and conversely, if $3\nmid [\mathcal{O}_{K}:\mathbb{Z}[\theta]],$ then $\bar F_{i}(x)\nmid \bar M(x),$ for all $i.$
			
			\textbf{Case 5.2:} Let $R=2.$ Then, from equation (\ref {eq26}), we obtain 
			\begin{align}\label{eq33}
				\begin{split}
					\bar M(x)&= \bigg(\displaystyle\prod_{i=1}^{l_{4}} \bar F_{i}(x)\bigg)\bar {B_{1}}(x)+\overline {\bigg(\frac {c}{(3k-1)(k-1)}\bigg)}\bigg[~~ \overline {2\bigg(\frac {a^{2}-1}{3}\bigg)}x^{3}+ x^{2}\\&~~~~~~~~~~~~~~~~~~~~~~~~~~~~~~~~~~~~~~~~~~~~~~~~~~~~~~~~~~~~~~~~~~~~~~~~~~~~~~+\overline {(2+ac)}x+\overline {\bigg(\frac {c^{2}-1}{3}\bigg)}~~\bigg]\\&=\bigg(\displaystyle\prod_{i=1}^{l_{4}} \bar F_{i}(x)\bigg)\bar {B_{1}}(x)+\overline {\bigg(\frac {c}{(3k-1)(k-1)}\bigg)}M_{4}(x)~ (\text {say}).
				\end{split}
			\end{align}
		From this, it is clear that $\bar f(x)$ and $\bar M(x)$ have a common zero in the algebraic closure of the field $\F_{3}$ if and only if $\bar f(x)$ and $\bar M_{4}(x)$ do as well. Thus, $\bar f(x)$ and $\bar M(x)$ have no common zeros if and only if $$\bigg[~~ \overline {2\bigg(\frac {a^{2}-1}{3}\bigg)}x^{3}+ x^{2}+\overline {(2+ac)}x+\overline {\bigg(\frac {c^{2}-1}{3}\bigg)}\bigg]$$ is co-prime to $\bar f(x).$ By considering the above cases (\textbf{5.1}) and (\textbf{5.2}) collectively and using the Dedekind  criterion (\ref {T4}), we complete the proof of the fifth part.

		\indent	Now, we consider the $\textbf{final part}$ when $p\nmid abc.$ If $p\nmid abc,$ from Lemma (\ref {l1}), we have $p\nmid n(n-1)(n-3)$ because of $nc=(n-3)a=(n-1)b$ that means $p\neq 2,~3$ (since $p\nmid n(n-1)$ and $3|n$ according to our hypothesis).  
		Let $p$ be an odd prime. Now, there are two possibilities that $\bar f(x)$ has repeated zeros or not. Assume $\zeta$ is a repeated zero of $\bar f(x).$ Then, $\bar f(\zeta)=\bar f'(\zeta)=\bar 0,$ where
		\begin{align}
			\label{eq34}
			\bar f(\zeta)=\zeta^{n}+\bar {a}\zeta^{3}+\bar b \zeta+\bar {c}=\bar 0
		\end{align}
		and
		\begin{align}
			\label{eq35}
			\bar f'(\zeta)=\bar {n}\zeta^{n-1}+3\bar a \zeta^{2}+\bar {b}=\bar 0.
		\end{align}
		From equation (\ref {eq35}), we have 
		\begin{align}
			\label{eq350}
			\zeta^{n-1}=-(\bar {n})^{-1}(3\bar a \zeta^{2}+\bar {b}).
		\end{align}
		By substituting the value of $\zeta^{n-1}$ in the equation (\ref {eq34}), we get 
		\begin{align}
			\label{eq36}
			&~~~~~~~~~~~\zeta [-(\bar {n})^{-1}(3\bar a \zeta^{2}+\bar {b})]+\bar {a}\zeta^{3}+\bar b \zeta+\bar {c}=\bar 0\nonumber  \\&~~\text {i.e.}~~-3\bar a \zeta^{3}-\bar {b}\zeta+\bar {n}\bar {a}\zeta^{3}+\bar {n}\bar b \zeta+\bar {n}\bar {c}=\bar 0\nonumber \\&~~\text {i.e.}~~~~~~~(\bar {n}-3)\bar {a} \zeta^{3}+(\bar {n}-1)\bar b\zeta+\bar {n}\bar {c}=\bar 0.
		\end{align}
		Applying Lemma (\ref {l1}) on the equation (\ref {eq36}), we have 
		\begin{align}
			\label{eq37}
		\bar {n}\bar {c}(\zeta^{3}+\zeta+1)=\bar 0~~\text {or}~~(\zeta^{3}+\zeta+1)=\bar 0,
		\end{align}
		because $p\nmid nc.$ Thus, $\bar f(x)$ has no repeated zeros if $(x^{3}+x+1)$ is co-prime to $\bar f(x).$

		If $(x^{3}+x+1)$ is not co-prime to $\bar f(x),$
		 then $\bar f(x)$ may have repeated zeros. Let $\bar f(x)=\displaystyle\prod_{i=1}^{l_{5}} (\bar q_{i}(x))^{e_{i}},$ where  $\bar q_{i}(x)$ be the distinct monic irreducible polynomial factors of $\bar f(x)$ over the field $\F_{p}$ and $q_{i}(x)$ are  respectively monic lifts, for all $i\in\{1,~2,~\ldots,l_{5}\}.$
		Define $M(x)$ as $$M(x)=\frac{1}{p}\bigg(f(x)-\displaystyle\prod_{i=1}^{l_{5}}  q_{i}(x)^{e_{i}}\bigg).$$

		If $\bar f(x)$ has no repeated zeros, then we are done.
		Let $\zeta$ is a repeated zero of $\bar f(x),$ then $\zeta$ is a zero of $(x^{3}+x+1).$ Thus, if $(x^{3}+x+1)$ is co-prime to $\bar M(x),$ then $\bar q_{i}(x)\nmid \bar M(x),$ where $\bar q_{i}(\zeta)=\bar 0.$ If $(x^{3}+x+1)$ is not co-prime to $\bar M(x),$ then $\bar M(\zeta)\neq \bar 0$ if and only if $\bar q_{i}(x)\nmid \bar M(x).$
		 Hence, by using (Theorem 6.1.4, \cite{HC}), we have complete proof of the final part.
		
	Also, the first part of the Lemma (\ref {l2}) implies that there does not exist any prime $p$ which satisfies the given hypothesis along with the conditions either $p|a$ and $p\nmid bc$ or $p|c$ and $p\nmid ab.$
		
			This completes the proof of the theorem.

		\end{proof}
			\begin{proof}[\bf{Proof of corollary \ref{C1}}]
		The proof of the corollary follows from the theorem (\ref {th1}). Indeed, if each prime $p$ divides $D_{f}$ and satisfies one of the following conditions from (1) to (6) of Theorem (\ref {th1}), then $p\nmid[\mathcal{O}_{K}:\mathbb{Z}[\theta]].$ Therefore, using the formula $$D_{f}= [\mathcal{O}_{K}:\mathbb{Z}[\theta]]^{2}D_{K},$$ we have $[\mathcal{O}_{K}:\mathbb{Z}[\theta]]=1$ implying that $\mathcal{O}_{K}=\mathbb{Z}[\theta].$ The converse of the corollary holds directly from Theorem (\ref {th1}). This completes the proof. 
	\end{proof}
	
		\begin{proof}[\bf{Proof of theorem \ref{th2}}]
		Let	\begin{align}\label{eq3000}
			\frac{d^ny}{dx^n} + a \frac{d^3y}{dx^3}+b\frac{dy}{dx} + c y = 0
		\end{align}
		be a differential equation with integer coefficients, where $\frac{a}{a-c}=k\in\N$ such that $n=3k>4,$ and $2ab=3ac-bc.$ Let $ \phi (z)=z^{n}+az^{3}+bz+c$ be the corresponding auxiliary irreducible polynomial of (\ref{eq3000}) and let $\theta$ be a root of $\phi (z).$ Indeed, if each prime $p$ divides $D_{\phi}$ and satisfies one of the following conditions from (1) to (6) of Theorem (\ref {th1}), then by using the formula $D_{f}= [\mathcal{O}_{K}:\mathbb{Z}[\theta]]^{2}D_{K},$ we have $\mathcal{O}_{K}=\Z[\theta],$ where $\mathcal{O}_{K}$ is the ring of integers in the algebraic number field $K=\Q(\theta).$ Also, $$\Z[\theta]=\{a_{0}+a_{1}\theta+a_{2}\theta^{2}+\ldots+a_{n-1}\theta^{n-1}~|~a_{i_{1}} \in \Z,~\text{for all}~~ i_{1}=0,~1,~\ldots,n-1\}.$$ Thus, all the roots of $\phi(z)=0,$ must be of the form $a_{0}^{(i)}+a_{1}^{(i)}\theta+a_{2}^{(i)}\theta^{2}+\ldots+a_{n-1}^{(i)}\theta^{n-1}=s_{i}$ (say), where $a_{j-1}^{(i)}$ are integers for all $i,~j=1,~2,\ldots,n.$ Hence, the general solution of the given differential equation $(\ref {eq300})$ is of the form
		\begin{align}
			\label{eq3010}
			y(x)&=
			\sum_{i=1}^{n}c_{i}\displaystyle\prod_{j=1}^{n} e^{a_{j-1}^{(i)}\theta^{j-1}x},
		\end{align}
		where $c_{i}$ are arbitrary real constants for all $i\in \{1,~2,~\ldots,~n\}.$ This completes the proof.
	\end{proof}

		\section{\bf Examples}	 
		In this section, we provide some examples that demonstrate the outcomes of our theorem. In the following examples, $K=\mathbb{Q}(\theta)$ be an algebraic number field corresponding to the algebraic integer $\theta$ with minimal polynomial $f(x)$ and $\mathcal{O}_{K}$ denotes the ring of algebraic integers of the number field $K.$
		
		\begin{example}
			Let $f(z)= z^{12}+44z^3+36z+33$ be a polynomial corresponding to the auxiliary equation of 
			\begin{align}\label{eq30000}
				\frac{d^{12}y}{dx^{12}} + 44\frac{d^3y}{dx^3}+36\frac{dy}{dx} + 33 y = 0
			\end{align}
			and let $\theta$ be a zero of it. Then, we have $D_{f}=-2^{24}.3^{24}.11^{11}.29.37.$ Here $2|a,$ $2|b,$ and $2\nmid c,$ where $a=44, b=36, c=33.$ Also, $2|u_{2},$ $2|u_{1},$ and $2\nmid u_{0},$ where $u_{0}=\frac{c+(-c)^{4}}{2},$ $u_{1}=\frac{b}{2},$ $u_{2}=\frac{a}{2},$ therefore using the section (i) of the part (2) of Theorem (\ref {th1}), we have $2\nmid  [\mathcal{O}_{K}:\mathbb{Z}[\theta]].$ Similarly, $3\nmid a,$ $3|b,$ $3| c,$ and $3|w_{2},~3|w_{1},~~3\nmid w_{0},$ $w_{0}=\frac{c}{3},$ $w_{1}=\frac{b}{3},$ $w_{2}=\frac{a+(-a)^{9}}{3},$ therefore using the section (i) of the part (4) of Theorem (\ref {th1}), we have $3\nmid  [\mathcal{O}_{K}:\mathbb{Z}[\theta]].$ Again, $11| a,$ $11\nmid b,$ $11| c,$ and $11\nmid v_{2}v_{1}v_{0},$  where $v_{0}=\frac{c}{11},$ $v_{1}=\frac{b+(-b)^{11}}{11},$ $v_{2}=\frac{a}{11},$ so by using the section (iii) of the part (3) of Theorem (\ref {th1}), we have $11\nmid  [\mathcal{O}_{K}:\mathbb{Z}[\theta]].$ Since $D_{f}= [\mathcal{O}_{K}:\mathbb{Z}[\theta]]^{2}D_{K},$ therefore by applying Theorem (\ref {th2}), the solution of the differential equation (\ref {eq30000}) is of the form
			\begin{align*}
				y(x)&=
				\sum_{i=1}^{12}c_{i}\displaystyle\prod_{j=1}^{12} e^{a_{j-1}^{(i)}\theta^{j-1}x},
			\end{align*}
			where $a_{j-1}^{(i)}$ are integers and $c_{i}$ are arbitrary real constants, for all $i,~j\in \{1,~2,~\ldots,~12\}.$
			
		\end{example}
		
		\begin{example}
			Let $f(x)= x^{9}+12x^3+9x+8$ be the minimal polynomial of the algebraic integer $\theta$ over the field $\mathbb{Q}.$ Then, we have $D_{f}=2^{24}.3^{19}.47.$ Here $2|a,$ $2|c,$ and $2\nmid b,$ where $a=12, b=9, c=8.$ Since, $2| v_{0}~ (v_{0}=\frac {c}{2})$ which implies that none of the sections of the part (3) of Theorem (\ref {th1}) are satisfies, therefore $2| [\mathcal{O}_{K}:\mathbb{Z}[\theta]].$ Thus, $K$ is not a monogenic field with respect to $\theta$.
		\end{example}
	\section{Acknowledgement} Karishan Kumar extends his gratitude to the CSIR fellowship for partial support under the file no: 
	09/1005(16567)/2023-EMR-I.

	\end{document}